\documentclass[reqno,a4paper]{amsart} 
\usepackage{amsmath,amsthm,amssymb,amsfonts,mathrsfs,color, enumerate}
\usepackage{latexsym,esint}

\usepackage{graphicx}
\usepackage{epsfig}
\usepackage{float}

\usepackage{pdfsync}

\usepackage{wasysym}

\usepackage{empheq}

\usepackage{url}

\newcommand{\RN}[1]{\textup{\uppercase\expandafter{\romannumeral#1}}}

\def\e{{\mathop {\rm e}\nolimits}}

\newcommand{\ep}{\varepsilon}

\newcommand{\R}{\mathbb{R}}
\newcommand{\N}{\mathbb{N}}

\def\ubar{\bar{u}}
\def\bb{\bar{\beta}}

\def\wt{\widetilde}
\def\tu{{\wt u}}
\def\ts{{\wt s}}

\def\tx{{\wt x}}
\def\tr{{\wt r}}
\def\tp{{\wt p}}
\def\tbeta{{\wt \beta}}
\def\tpsi{{\wt \psi}}

\newtheorem{theorem}{Theorem}[section]

\newtheorem{proposition}[theorem]{Proposition}

\theoremstyle{definition}
\begingroup

\newtheorem{remark}[theorem]{Remark}

\endgroup

\renewcommand{\theta}{\vartheta}

\numberwithin{equation}{section}

\title[Optimal control of the transmission rate in compartmental epidemics]
{Optimal control of the transmission rate\\ in compartmental epidemics}
\author[L. Freddi]{Lorenzo Freddi}

\address[L. Freddi]{Dipartimento di Scienze Matematiche, Informatiche e Fisiche,
via delle Scienze 206, 33100 Udine, Italy
}\email{lorenzo.freddi@uniud.it}

\date{}

\subjclass[2010]{49J45, 37N25, 	92D30}
\keywords{Optimal Control, Calculus of Variations, Compartmental Epidemics}

\begin{document}

\begin{abstract}
We introduce a general system of ordinary differential equations that includes some classical and recent  models for the epidemic spread in a closed population without vital dynamic in a finite time horizon.  
The model is vectorial, in the sense that it accounts for a vector valued state function whose components represent various kinds of exposed/infected subpopulations, with a corresponding vector of control functions possibly different for any subpopulation. 
In the general setting, we prove well-posedness and positivity of the initial value problem for the system of state equations and
the existence of solutions to the optimal control problem of the coefficients of the nonlinear part of the system, under a very general cost functional. We also prove the uniqueness 
of the optimal solution for a small time horizon when the
cost is superlinear in all control variables with possibly different exponents in the interval $(1,2]$. 
We consider then a linear cost in the control variables and study the singular arcs. Full details are given in the case $n=1$
and the results are illustrated by the aid of some numerical simulations.
\end{abstract}

\maketitle

\
\section{Introduction}

Since the introduction of the first compartmental epidemic mo\-del by Kermack and McKendrick \cite{KMK1927}
and the subsequent extensions and generalizations (\cite{AM1991,DE2000,Heth2000,HR2004}),
optimal control problems for such models have been studied in order to reduce the economics, social and treatment costs 
of the epidemic spread (\cite{FLMcN1998,Behncke2000,Gumel04,YZ2008,GS2009,YZ2009,LS2011,HD2011,MK2012,GH2018}). Most of these works aimed to control the coefficients of the linear part 
of the differential equations to model isolation, quarantine and vaccination effects.
Control problems of the transmission coefficients, that is of the nonlinear part of the differential equations, 
have been considered mainly after the SARS-CoV epidemic of 2003 (\cite{LML2010,KK2014,SM2017,BJ2018}) and a recent renewed interest is due 
to the SARS-CoV-2 pandemic of 2019-2020 (\cite{GBetal2020,KS2020,LEE2013310,TLSB2020}).
The transmission rate can be, indeed, reduced by means of social distance policies.

In this paper we introduce a general setting that includes many of the mentioned models and possibly other different kind of epidemics in a closed population without vital dynamic in a finite time horizon $I:=[0,t_f]$. It is given by a set of ordinary differential equations
of the form
$$
\begin{cases}
\dot s(t)=-s(t)\,\beta(t)\cdot x(t)+\rho r(t)\\
\dot x(t)=s(t)\,\beta(t)\cdot x(t)\e_1+M x(t)\\
\dot r(t)=\sigma\cdot x(t)-\rho r(t)\\
\dot d(t)=\mu\cdot x(t)
\end{cases}
$$
where $M=(m_{ij})$ is a quasimonotone (or Metzler) lower triangular matrix, that is a lower triangular square matrix whose elements out of the diagonal are nonnegative.
 
As usual, $\cdot$ denotes the scalar product, $\e_1=(1,0,\dots0)$ is the first vector of the canonical basis of $\R^n$ and $M x$ denotes the usual row-by-column multiplication of the matrix $M$ with the column vector $x$. To model the evolution of an epidemic
\begin{itemize}
\item $s$ is the scalar density of the susceptible population, $x$ is the $n$-vector of the densities of various kind of infected populations (exposed, asimptomatic, infected, etc.) and $r$ and $p$ are the scalars of recovered and deceased individuals, respectively;
\item $\beta\in L^\infty(I;[0,1]^n)$, $\sigma,\mu\in[0,1]^n$, $\rho\in[0,1]$, $M\in [0,1]^{n\times n}$, 
are prescribed coefficients with various epidemiological meanings. Namely, $\beta$ is the vector-function of transmission coefficients, 
$\sigma$ and $\mu$ are constant vectors representing the fraction of recovered and dead individuals for any subpopulation, respectively, $\rho$ represents the fraction of recovered population that become susceptible again and $M$ represents 
the fraction of individuals that pass from a subpopulation to another after a certain time (for instance the exposed that becomes sintomatic). 
\end{itemize}
A specific feature of the model is that it is {\em vectorial}, in the sense that it accounts for a vector valued state function $x$ 
whose components represent various kinds of exposed/infected subpopulations, with a corresponding vector of control functions   possibly different for any subpopulation. 
Our general setting includes several classical models, like  
\begin{itemize}
\item SIR, SIRS, SIRD in the case $n=1$,
\item SEIR, SEIRS in the case $n=2$.
\end{itemize}
Besides these classical ones, many other models fall in the general setting; among the most recent we have for instance:
\begin{itemize}
\item a model for COVID-19 epidemic given in \cite{GBetal2020}, $s=S$, $x=(I,D,A,R,T)$ (that is there are $n=5$ subpopulations of exposed/infected individuals), $r=H$, $p=E$, $\beta_1=\alpha$, $\beta_2=\beta$, $\beta_3=\gamma$, $\beta_4=\delta$, $\beta_5=0$, $\rho=0$, $\sigma_1=\lambda$, $\sigma_2=\rho$, $\sigma_3=\kappa$, $\sigma_4=\xi$, $\sigma_5=\sigma$, $\mu_1=\mu_2=\mu_3=\mu_4=0$, $\mu_5=\tau$ and 
$$
M=\left(\begin{matrix}
-(\ep+\zeta+\lambda)&0&0&0&0\\
\ep&-(\eta+\rho)&0&0&0\\
\zeta&0&-(\theta+\mu+\kappa)&0&0\\
0&\eta&\theta&-(\nu+\xi)&0\\
0&0&\mu&\nu&-(\sigma+\tau)
\end{matrix}\right)
$$ 
\item a model for the optimal control of COVID-19 outbreak given in \cite{TLSB2020}, where $x=(e,a,i)$ (that is there are $n=3$ subpopulations of exposed/infected individuals), $\beta_1=0$, $\beta_2=\alpha_a/N$, $\beta_3=\alpha_i/N$, $\rho=\gamma$, $\sigma_1=0$, $\sigma_2=\rho$, $\sigma_3=\beta$, $\mu_1=\mu_2=0$, $\mu_3=\mu$ and
$$
M=\left(\begin{matrix}
-t^{-1}_{latent}&0&0\\
t^{-1}_{latent}&-(\kappa+\rho)&0\\
0&\kappa&-(\beta+\mu)
\end{matrix}\right)
$$
\item a model for the optimal control of influenza given in \cite{LEE2013310} where, in the basic formulation, $x=(e,i,a)$ (that is there are $n=3$ subpopulations of exposed/infected individuals), $\beta_1=\ep$, $\beta_2=1-q$, $\beta_3=\delta$, $\rho=0$, $\sigma_1=0$, $\sigma_2=f\alpha$, $\sigma_3=\eta$, $\mu_1=0$, $\mu_2=f$, $\mu_3=0$, and
$$
M=\left(\begin{matrix}
-\kappa&0&0\\
p\kappa&-\alpha&0\\
(1-p)\kappa&0&-\eta
\end{matrix}\right).
$$
\end{itemize}
In our analysis we assume that the time $t$ belongs to a finite time horizon $I:=[0,t_f]$ where the final time $t_f>0$ is given.
In the general setting, we prove the well-posedness of the initial value problem for the system of state equations.
The existence of solutions to the optimal control problem under a very general cost functional is a standard matter.
On the contrary, the problem of uniqueness of the optimal solution has received much less attention.
In 1998 Fister \cite{FLMcN1998} proved the uniqueness of the solution for a control problem of the chemotherapy in {HIV}
for a sufficiently small time horizon and a cost funtional that is quadratic in the control variable. Our general problem does not fall into the same setting, so that Fister's result cannot be 
directly applicated. Nevertheless, the idea can be fruitfully used also in our framework 
leading to the same kind of uniqueness result which, on the other hand, can be extended to the case in which the
cost is superlinear in all control variables with possibly different exponents in the interval $(1,2]$; this allows
to capture a nonlinear growth of costs due to overcrowding in healthcare facilities and to gradually higher level of slowdown of the economy, with different degrees of nonlinearity associated to different distance and slowdown policies that are simultaneously actuated.
It is important to remark that this uniqueness result for a small time horizon cannot be iterated in order to obtain 
a uniqueness result for every $t_f$ (see Remark \ref{unique}): this problem is still open.

In the last section of the paper we consider a linear cost in the control variables and study the singular arcs. Full details are given in the case $n=1$
together with a few numerical simulations made by using the package Bocop \cite{Bocop, BocopExamples}.

\section{Well-posedness of the initial value problem}

Let us remark that, under differentiability of the population densities, 
the total population is preserved if and only if
\begin{eqnarray*}
0&=&\dot s+\sum_{i=1}^n\dot x_i+\dot r+ \dot d\\
&=&\sum_{h=1}^n\big(\sum_{i=1}^nm_{ih}+\sigma_h+\mu_h\big)x_h.
\end{eqnarray*}
For this reason we assume that the coefficients of the system satisfy the {\em closed population assumption}  
\begin{equation}\label{mc}
\sum_{i=1}^nm_{ih}+\sigma_h+\mu_h=0\quad\mbox{ for } h=1,...,n.
\end{equation}
With this hypothesys and under initial conditions satisfying the requirement
$$
s(0)+\sum_{i=1}^nx_{i}(0)+r(0)+d(0)=1
$$
then we have 
\begin{equation*}\label{cmass}
s(t)+\sum_{i=1}^nx_i(t)+r(t)+d(t)=1\quad \forall\,t\in I.
\end{equation*}
The closed population assumption is a condition on the coefficients of the system (hence independent of the evolution of any subpopulation) that is sufficient to ensure that the total population is preserved. Physically, it represents a mass conservation property. It is satisfied by the epidemic models \cite{GBetal2020}, \cite{LEE2013310} and \cite{TLSB2020} mentioned in the introduction.

Under the closed population assumption, by the previous equation, the evolution of $d(t)$ can be directly 
deduced by those of the other subpopulations. Then, the fourth equation can be eliminated from the system and we deal with the following reduced 
initial value problem:
\begin{equation}\label{sxr}
\begin{cases}
\dot s(t)=-s(t)\,\beta(t)\cdot x(t)+\rho r(t)\\
\dot x(t)=s(t)\,\beta(t)\cdot x(t)\e_1+M x(t)\\
\dot r(t)=\sigma\cdot x(t)-\rho r(t)\\
s(0)=s_0,\ x(0)=x_0,\ r(0)=r_0.
\end{cases}
\end{equation}
Since $x$ is a vector then, of course, $x_0=(x_{01},\dots,x_{0n})$. To be consistent with the epidemiological character of 
the model, we make the following {\em initial condition assumption}
\begin{eqnarray}
&&s_0,r_0\in[0,1],\  x_0\in[0,1]^n,\nonumber\\
&&s_0+\sum_{i=1}^nx_{0i}+r_0\le1,\label{wpa}\\ 
&&x_{01}>0.\nonumber 
\end{eqnarray}

\begin{theorem}\label{wpth}
Let us assume that $\beta\in L^\infty(I;[0,1])$, $\rho\in[0,1]$, $M\in [0,1]^{n\times n}$ be a
lower triangular quasimonotone matrix and $\sigma,\mu\in[0,1]^n$  satisfy the closed population assumption \eqref{mc} and the initial condition assumption \eqref{wpa}.
 Then the system
\eqref{sxr} admits a unique solution $(s,x,r)$ such that
\begin{enumerate}
\item the solution is Lipschitz continuous on the interval $I$ and taking values 
$x(t)\in[0,1]^n$ and $s(t),r(t)\in[0,1]$ for every $t\in I$,
\item if $s_0>0$ then $s(t)>0$ for every $t\in I$,
\item if $r_0>0$ then $r(t)>0$ for every $t\in I$,
\item if $x_{0i}>0$ then $x_i(t)>0$ for every $t\in I$, $i=1,\dots,n$.
\end{enumerate}
\end{theorem}

\begin{proof}
Since the dynamic is locally Lipschitz, then it is classical that we have local existence and uniqueness of an absolutely continuous solution (see for instance \cite[I.3]{Hale1980}). Let $[0,\tau)$, $\tau\le t_f$, be an interval in which the solution exists.
By continuity of $x_1$ and since $x_1(0)>0$  we can also assume that $x_1>0$ in $[0,\tau)$.

Since $M$ is lower triangular, then
$$
\dot x_2=m_{21}x_1+m_{22}x_2
$$
and since $m_{21}\ge0$ then
$$
\dot x_2\ge m_{22}x_2\quad\mbox{ on }[0,\tau).
$$
This readily implies that $x_2\ge0$ on $[0,\tau)$ (strictly positive if $x_{02}>0$). 
Iterating the procedure and using the properties of $M$, we have that $x_i\ge0$ on $[0,\tau)$ (strictly positive if $x_{0i}>0$) for $i=1,\dots,n$.

Then we have
$$
\dot r\ge-\rho r
$$
which implies $r\ge 0$ on $[0,\tau)$ (strictly positive if $r_{0}>0$). 

Finally, by integration, 
$$
s(t)=\e^{-\int_0^t\beta(\xi)\cdot x(\xi)\,d\xi}\Big(\int_0^t\e^{\int_0^\xi\beta(\tau)\cdot x(\tau)\,d\tau}
\rho r(\xi)\,d\xi+s_0\Big)
$$
which implies that $s(t)\ge0$ in $[0,\tau)$ (strictly positive if $s_0>0$). 

Since the assumptions on the coefficients ensure that the total population is preserved, then we immediately have that
$s(t),r(t)\in[0,1]$, and $x_i(t)\in[0,1]^n$ for $i=1,\dots,n$, for every $t\in[0,\tau)$. Hence the solution can be continued 
and we have global existence of an absolutely continuous solution on $I$ satisfying {\it 2-4}. 
Consequently, by the equations of the system we have that also the derivatives are bounded implying the Lipschitz continuity of the solution.
\end{proof}

\begin{remark}{
The proof works also if $I=[0,+\infty)$.}
\end{remark}

\section{Optimal control}

We aim here to study the optimal control of the system of ODEs under social distance. That is, 
we take 
$$
\beta(t):=\bb-u(t)
$$
where  $u$ is a vectorial control variable. Since it is introduced to reduce 
the transmission rates, then it is natural to require that $u$ belong to a space of bounded functions, like
the space $L^\infty(I;K)$ of (equivalence classes of) Lebesgue measurable functions defined on $I$ and taking values in $K$ up to a set of measure zero,
with 
$$
K=\prod_{i=1}^n[0,\ubar_i],\quad \ubar_i\in(0,\bb],\quad \bb\in(0,1).
$$ 
Here $\bb$ represents the vector of transmission coefficients without any control. The role of the control vector variable $u$ is then to reduce the transmission rates by various levels of social distance, slowdown of the economy, isolation and quarantine measures. 
The value of $\ubar$ depends on the distance policies that can be put into being. 
The choice of $\ubar=\bb$ means that we are able to impose rules that completely stop transmission, and this is compatible only with isolation strategies, but could be unrealistic for other kind of measures. 

The optimal control problem consists in minimizing a {\em cost functional} of the form 
\begin{equation}\label{cost}
J(x,u)=\int_0^{t_f}f_0(t,x,u)\,dt
\end{equation}
where $f_0$ is a given {\em running cost}, under the set of 
{\em state equations} 
\begin{equation}\label{stateeq}
\begin{cases}
\dot s(t)=-s(t)\,\big(\bb-u(t)\big)\cdot x(t)+\rho r(t)\\
\dot x(t)=s(t)\,\big(\bb-u(t)\big)\cdot x(t)\e_1{+}M x(t)\\
\dot r(t)=\sigma\cdot x(t)-\rho r(t)\\
s(0)=s_0,\ x(0)=x_0,\ r(0)=r_0
\end{cases}
\end{equation}
satisfying the 
{\em initial condition assumption} \eqref{wpa}
and  the {\em closed population assumption} \eqref{mc}.
The cost functional $J$ represents the cost of treatments and hospitalization for the populations $x$ of exposed/infected individuals 
and its dependence on $u$ allows to capture 
the economic and social cost of slowdown, isolation, quarantine and social distance measures in general.

\subsection{Existence of an optimal solution}

An {\em optimal solution} to the control problem \eqref{cost}-\eqref{stateeq} is a vector function 
$(u,s,x,r)\in L^\infty(I;K)\times W^{1,\infty}(I)\times W^{1,\infty}(I;\R^n)\times W^{1,\infty}(I)$ that minimizes the cost $J$ and satisfies the set of state equations. 

In the definition above, $W^{1,\infty}(I)$ denotes the usual Sobolev space of functions that are essentially bounded together with the first distributional derivative, while $W^{1,\infty}(I;\R^n):=\big(W^{1,\infty}(I)\big)^n$.

The following existence theorem for a very general cost functional holds.

\begin{theorem}
If $f_0:(0,t_f)\times\R^n\times\R^n\to[0,+\infty)$ is a normal convex integrand, that is 
it is measurable with respect to the Lebesgue $\sigma$-algebra in $(0,t_f)$ and the Borel $\sigma$-algebra in $\R^n\times\R^n$ 
and there exists a subset $N$ of $(0,t_f)$ of Lebesgue measure zero such that  
\begin{enumerate}
\item $f_0(t,\cdot,\cdot)$ is lower semicontinuous for every $t\in(0,t_f)\setminus
N$,
\item $f_0(t,x,\cdot)$ is convex for every $t\in(0,t_f)\setminus
N$ and $x\in\R^n$,
\end{enumerate}
then there exists an optimal solution $(u,s,x,r)$ to the control problem \eqref{cost}-\eqref{stateeq}.
\end{theorem}

To prove the existence of an optimal solution we could invoke some very general theorems, like Theorem 23.11 of \cite{Clarke2013}, that can be applied to a lot of other situations. To be self contained and since it will become useful in the sequel,
we prefer to sketch here a direct proof based on the observation that it is equivalent to prove the existence of a minimizer
of the functional 
\begin{equation}\label{minF}
F(u,s,x,r):=J(x,u)+\chi_{\Lambda}(u,s,x,r)
\end{equation}
where $\Lambda$ is the set of admissible pairs, that is all state-control vectors $(u,s,x,r)$ that satisfy the initial value problem \eqref{stateeq}, while $\chi_\Lambda $ denotes the indicator function of $\Lambda$ that takes the value $0$ on $\Lambda$ and $+\infty$ otherwise.

\vspace{1ex}
 
\begin{proof} On the domain of $F$, that is the space $L^\infty(I;K)\times W^{1,\infty}(I)\times W^{1,\infty}(I;\R^n)\times W^{1,\infty}(I)$ we consider the topology given by the product of the weak* topologies of the four spaces and aim to prove sequential lower semicontinuity and coercivity  
of the functional $F$ with respect to this topology. By the 
Direct Method of the Calculus of Variations, these properties imply the existence of a solution to the minimum problem. They are direct consequences of the  
fact that the space of control is weakly* compact, that the assumptions on $f_0$ imply that the cost functional $J$ is weakly* lower semicontinuous 
(which is a particular case of De Giorgi and Ioffe's Theorem; see for instance \cite[Theorem 7.5]{FL_MMCVLp}) and the fact that $\Lambda$ is closed with respect to the weak* convergence. 
\end{proof}

\begin{remark}{
The requirement on $f_0=f_0(t,x,u)$ to be a normal convex integrand is satisfied, in particular, if it is a piecewise continuous function of $t$, continuous in $x$ and convex in $u$. Assumptions of this kind are usually satisfied in the applications.}
\end{remark}

\section{Optimality conditions} 

To write necessary conditions of optimality we require that $f_0$ satisfies the classical regularity assumption  $f_0\in C^1([0,t_f]\times[0,1]^n\times[0,\bb])$ and be nonnegative. 
Let us introduce the adjoint variables $p_0\ge0$, $p_s\in\R$, $p_x=(p_{x_1},...,p_{x_n})\in\R^n$ and the Hamiltonian
$$
H(t,u,s,x,r,p_0,p_s,p_x)=p_0f_0+p_s f_s+p_x\cdot f_x+p_r\cdot f_r
$$
where $f_s=-s\,(\bb-u)\cdot x+\rho r$, $f_x=s\,(\bb-u)\cdot x\,\e_1+M x$, $f_r=\sigma\cdot x-\rho r$   are the dynamics of the state equations. 
After some manipulations,  the Hamiltonian turns out to be
\begin{eqnarray*}
&&H(t,u,s,x,r,p_0,p_s,p_x,p_r)=\\
&&\hspace{4ex}=p_0f_0(t,x,u)+(p_{x_1}-p_s) s(\bb-u)\cdot x+\rho (p_s-p_r) r+ p_x\cdot M x+p_r\,\sigma\cdot x.
\end{eqnarray*}
By Pontryagin's theorem (see for instance \cite[Section IV.22]{Clarke2013}, \cite[Section 2.2.2]{SLbook}), given an optimal solution $(u,s,x,r)$, there exist a nonnegative constant $p_0$ and absolutely continuous adjoint (or conjugate) state functions (or costates) $p_s$, $p_x$ and $p_r$ that satisfy the non-degeneration property
\begin{equation}\label{ndc}
\big(p_0,p_s(t),p_x(t),p_r(t))\ne0\quad\forall\,t\in[0,t_f]
\end{equation}
and such that
\begin{eqnarray*}
&&H\big(t,u(t),s(t),x(t),r(t),p_0,p_s(t),p_x(t),p_r(t)\big)=\\
&&\hspace{12ex}\displaystyle=\inf_{u\in K}\!H\big(t,u,x(t),r(t),p_0,p_s(t),p_x(t),p_r(t)\big)
\end{eqnarray*}
for almost every $t\in[0,t_f]$. 
This is a minimum problem for a continuous function of $n$ real variables on a compact set. 
To solve it explicitly we should prescribe the running cost $f_0$. 

The adjoint states $p_s$, $p_x$ and $p_r$ must solve the adjoint (or conjugate) equations 
\begin{equation*}
\left\{
\begin{aligned}
&\dot p_s=-\frac{\partial H}{\partial s}\\
&\dot p_{x}=-\frac{\partial H}{\partial {x}}\\
&\dot p_r=-\frac{\partial H}{\partial r}\\
\end{aligned}
\right.
\end{equation*}
where $\frac{\partial }{\partial x}:=(\frac{\partial }{\partial x_i})_{i=1,...,n}$, that is
\begin{equation*}
\begin{cases}
\dot p_s=-(p_{x_1}-p_s)(\bb-u)\cdot x\\[1ex]
\displaystyle\dot p_{x}=-p_0\frac{\partial f_0}{\partial x}(t,x,u)-(p_{x_1}-p_s)s(\bb-u){-}M^Tp_{x}-p_r\sigma\\
\dot p_r=\rho (p_s-p_r)
\end{cases}
\end{equation*}
and have to satisfy the transversality conditions
\begin{equation}\label{aetc}
p_s(t_f)=p_{x_i}(t_f)=p_r(t_f)=0
\end{equation}
coming from the fact the final states are free.

\begin{remark}\label{lcp}{
By the non-degeneration property \eqref{ndc}, the transversality conditions $p_s(t_f)$ $= p_{x_i}(t_f)=p_r(t_f)=0$ imply that $p_0>0$.
Thus, without loss of generality, we can assume from now on that $p_0=1$.}
\end{remark}

\begin{remark}\label{bas}{
Since $f_0$ is $C^1$, then $\frac{\partial f_0}{\partial x}$ is continuous and hence bounded on $[0,t_f]$. By the adjoint equation it then follows that the adjoint states are Lipschitz continuous.}
\end{remark}

If the integrand $f_0$ is time independent, that is $f_0=f_0(x,u)$, then also the Hamiltonian is time independent and therefore it is constant along the optimal solutions, that is, there exists a constant $k$ such that
\begin{equation*}
f_0(x,u)+\big(p_{x_1}-p_s\big) s\big(\bb-u\big)\cdot x+\gamma \big(p_s-p_r\big) r+ p_x\cdot M x+p_r\sigma\cdot x=k
\end{equation*}
on the interval $[0,t_f]$. 

\vspace{2ex}

In the next sections we consider particular cost functionals in which the state and control variables are separated.
From the point of view of the solutions, the optimal control problem exhibits very different behaviors depending on how the cost grow with 
the control variable.

\section{Cost with a superlinear growth in the control variable}\label{sssup}

Let us consider now the case of a running cost of the form
\begin{equation}\label{f0}
f_0(t,x,u)=\nu(t,x)+\sum_{i=1}^nC_i u_i^{q_i}
\end{equation}
where $\nu\in C^1([0,t_f]\times[0,1]^n)$ is a non negative function, $C_i$ are strictly positive constants and $q_i>1$ for $i=1,\dots,n$.
A remarkable particular case is the quadratic one, in which $q_i=2$ for every $i$.

These assumptions allow to capture a nonlinear growth of costs due to overcrowding in healthcare facilities and to gradually higher level of slowdown of the economy, with various degrees of nonlinearity. The different constants and different exponents allow to prescribe different costs to different distance and slowdown policies that are simultaneosly actuated.

The Hamiltonian is
\begin{eqnarray*}
&&\hspace{-4ex}H(t,u,s,x,r,p_s,p_x,p_r)=\\
&&\hspace{-2ex}\displaystyle=\nu(t,x)+\sum_{i=1}^nC_i u_i^{q_i}+(p_{x_1}-p_s) s(\bb-u)\cdot x+\rho (p_s-p_r) r+ p_x\cdot M x+p_r\,\sigma\cdot x.
\end{eqnarray*}
The minum problem for the function 
$$
u\mapsto H(t,u,s,x,r,p_s,p_x,p_r)
$$
on the compact set  $K=\prod_{i=1}^n[0,\ubar_i]$ is easy to solve. The critical interior points must satisfy
$$
\frac{\partial H}{\partial u_i}=C_iq_iu_i^{q_i-1}-(p_{x_1}-p_s){sx_i}=0\ \iff\ u_i^{q_i-1}=\frac{1}{q_iC_i}(p_{x_1}-p_s){sx_i}\,.
$$
Hence, setting 
$$
\psi_i(t):=\frac{1}{q_iC_i}\big(p_{x_1}(t)-p_s(t)\big)s(t)x_i(t),
$$
the optimal control is characterized by the following componentwise conditions
\begin{eqnarray}\label{cuSIR}
u_i(t)
&=&\min\{\psi_i^+(t)^\frac{1}{q_i-1},\ubar_i\}\nonumber\\
&=&\begin{cases}
0&\mbox{ if }\psi_i(t)\le0,\\[0ex]
\psi_i(t)^\frac{1}{q_i-1}&\mbox{ if }\psi_i(t)\in(0,\ubar_i^{q_i-1}),\\[0ex]
\ubar_i&\mbox{ if }\psi_i(t)\ge\ubar_i^{q_i-1}
\end{cases}
\end{eqnarray}
where $\psi_i^+(t):=\max\{\psi_i(t),0\}$.

\begin{proposition}\label{qcuLc}
Any optimal control $u$ is Lipschitz continuous on $[0,t_f]$ and satisfies the final condition $u(t_f)=0$.
\end{proposition}

\begin{proof} It follows by the previous characterization of the optimal control and by the fact that the  
states and the costates are Lipschitz continuous functions. The final condition follows by the fact that 
the transversality conditions imply that $\psi_i(t_f)=0$, $i=1,\dots,n$.
\end{proof}

The adjoint states $p_s$, $p_x$ and $p_r$ must solve the adjoint equations and transversality conditions 
\begin{equation}\label{eaSIRsdnu}
\left\{
\begin{aligned}
&\dot p_s=-\eta(\bb-u)\cdot x\\
&\dot p_{x}=-\frac{\partial \nu}{\partial x}(t,x)-(p_{x_1}-p_s)s(\bb-u)-M^Tp_{x}-p_r\sigma\\
&\dot p_r=\rho (p_s-p_r)\\[1ex]
&p_s(t_f)=p_{x_i}(t_f)=p_r(t_f)=0.
\end{aligned}
\right.
\end{equation}

\subsection{Uniqueness of the optimal solution}

The problem of uniqueness of the optimal solution is of great importance in applications and nevertheless it is not 
a trivial question because of the nonlinearity of the state equations that lead to a lack of convexity of the functional $F=J+\chi_\Lambda$ (see \eqref{minF}) even if the cost is strictly convex.  

Nevertheless, we are able to prove the uniqueness of the solution when the cost is superlinear in all control variables with exponents $q_i\in(1,2]$. Moreover, the result holds only for a sufficiently small time horizon. The basic idea of the proof is due to Fister \cite{FLMcN1998} where, on the other hand, only the case $q_i=2$ is considered and for a control problem (for the chemotherapy in {HIV}) that does not fall into our abstract setting. 

Using the previuos discussion, we have that any optimal solution must solve the {\em optimality system}
given by the boundary value problems for the state and adjoint equations, and the characterization of the optimal control, that is 
\begin{equation}\label{os}
\begin{cases}
\dot s=-s\,(\bb-u)\cdot x+\rho r\\
\dot x=s\,(\bb-u)\cdot x\e_1+M x\\
\dot r=\sigma\cdot x-\rho r\\
\dot p_s=-(p_{x_1}-p_s)(\bb-u)\cdot x\\
\displaystyle\dot p_{x}=-\frac{\partial \nu}{\partial x}-(p_{x_1}-p_s)s(\bb-u)-M^Tp_{x}-p_r\sigma\\
\dot p_r=\rho (p_s-p_r)\\
s(0)=s_0,\ x(0)=x_0,\ r(0)=r_0\\
p_s(t_f)=p_{x_i}(t_f)=p_r(t_f)=0\\
\displaystyle u_i(t)=\min\big\{\max\{\frac{(p_{x_1}(t)-p_s(t))s(t)x_i(t)}{q_iC_i},0\}^\frac{1}{q_i-1},\ubar_i\big\},\quad i=1,\dots,n.
\end{cases}
\end{equation}

Using the optimality system we can prove the following uniqueness result.  
 
\begin{theorem} Let the running cost take the form \eqref{f0} with $q_i\in(1,2]$ for $i=1,\dots,n$ and $\nu\in C^1([0,t_f]\times[0,1]^n)$ non negative
and with Lipschitz continuous partial derivatives with respect to $x$ with a $t$-independent Lischitz constant, that is, there exists $L\ge0$ such that
\begin{equation}\label{lcnu}
|\frac{\partial \nu}{\partial x}(t,y)-\frac{\partial \nu}{\partial x}(t,z)|\le L|y-z|\qquad\forall\, x,y\in[0,1]^n, t\in[0,t_f].
\end{equation}
If $t_f$ is small enough than the optimal solution is unique.
\end{theorem}

\begin{proof} Let us assume that $(u,s,x,r)$ and $(\tu,\ts,\tx,\tr)$ are two optimal solutions of the control problem. Then 
  $(u,s,x,r,p_s,p_x,p_r)$ and $(\tu,\ts,\tx,\tr,\tp_s,\tp_x,\tp_r)$ are two solutions of the optimality system 
\eqref{os}. 

To be more contained, it will be useful in the sequel of the proof 
to go back to the shorter notation $\beta=\bb-u$ and $\tbeta=\bb-\tu$.

Inspired by \cite{FLMcN1998}, let us introduce for any $\lambda\ge0$ the functions
$$
s^{\lambda}:=\e^{-\lambda t}s,\quad x^{\lambda}:=\e^{-\lambda t}x,\quad r^{\lambda}:=\e^{-\lambda t}r,
$$
$$
p_s^{\lambda}:=\e^{\lambda t}p_s,\quad p_x^{\lambda}:=\e^{\lambda t}p_x,\quad p_r^{\lambda}:=\e^{\lambda t}p_r,
$$
and the analogous ones with the\ $\tilde{}$ variables. 

Substituting in the optimality system we obtain
the family of equivalent systems (one for every $\lambda$)
\begin{equation*}\label{osl}
\begin{cases}
\dot s^\lambda+\lambda s^\lambda=-\e^{\lambda t}s^{\lambda}\,\beta\cdot x^{\lambda}+\rho r^{\lambda}\\
\dot x^\lambda+\lambda x^\lambda=\e^{\lambda t}s^{\lambda}\,\beta\cdot x^{\lambda}\e_1+M x^{\lambda}\\
\dot r^\lambda+\lambda r^\lambda= \sigma\cdot x^{\lambda}-\rho r^{\lambda}\\
\dot p_s^\lambda-\lambda p_s^\lambda=-\e^{\lambda t}(p^{\lambda}_{x_1}-p^{\lambda}_s)\beta\cdot x^{\lambda}\\
\dot p_{x}^\lambda-\lambda p_{x}^\lambda =-\e^{\lambda t}\frac{\partial \nu}{\partial x}-\e^{\lambda t}(p^{\lambda}_{x_1}-p^{\lambda}_s)s^{\lambda}\beta
-M^Tp^{\lambda}_{x}-p^{\lambda}_r\sigma\\
\dot p_r^\lambda-\lambda p_r^\lambda=\rho (p^{\lambda}_s-p^{\lambda}_r)\\
s^{\lambda}(0)=s_0,\ x^{\lambda}(0)=x_0,\ r^{\lambda}(0)=r_0\\
p^{\lambda}_s(t_f)=p^{\lambda}_{x_i}(t_f)=p^{\lambda}_r(t_f)=0
\end{cases}
\end{equation*}
and the analogous one with the\ $\tilde{}$ variables. 
We start by considering the equations corresponding to the state $x$ and its conjugate $p_x$, that is 
$$
\begin{cases}
\dot x^\lambda+\lambda x^\lambda=\e^{\lambda t}s^{\lambda}\,\beta\cdot x^{\lambda}\,\e_1+M x^{\lambda}\\
\dot p_{x}^\lambda-\lambda p_{x}^\lambda =-\e^{\lambda t}\frac{\partial\nu}{\partial x}-\e^{\lambda t}(p^{\lambda}_{x_1}-p^{\lambda}_s)s^{\lambda}
\beta-M^Tp^{\lambda}_{x}-p^{\lambda}_r\sigma\\
s^{\lambda}(0)=s_0,\ x^{\lambda}(0)=x_0,\ r^{\lambda}(0)=r_0\\
p^{\lambda}_s(t_f)=p^{\lambda}_{x_i}(t_f)=p^{\lambda}_r(t_f)=0.
\end{cases}
$$
Subtracting side by side,
scalarly multiplying the first equation by $x^\lambda-\tx^\lambda$ and the second by $p_x^\lambda-\tp_x^\lambda$, and integrating with the usage of the boundary conditions, we obtain
\begin{eqnarray*}
&&\dfrac{\big|x^\lambda(t_f)-\tx^\lambda(t_f)\big|^2}{2}+\lambda \int_0^{t_f}|x^\lambda-\tx^\lambda|^2dt=\\
&&\hspace{8ex}\displaystyle=\int_0^{t_f}\e^{\lambda t}(x_1^\lambda-\tx_1^\lambda)(s^{\lambda}\beta\cdot x^{\lambda}-\ts^{\lambda}\tbeta\cdot \tx^{\lambda})+(x^{\lambda}-\tx^{\lambda})\cdot M (x^{\lambda}-\tx^{\lambda})\,dt,
\end{eqnarray*}
\begin{eqnarray*}
&&\displaystyle\frac{\big|p_{x}^\lambda(0)-\tp_{x}^\lambda(0)\big|^2}{2}+\lambda \int_0^{t_f}|p_{x}^\lambda-\tp_{x}^\lambda|^2dt=\\ 
&&\hspace{13ex}\displaystyle=\int_0^{t_f}\e^{\lambda t}(p_{x}^\lambda-\tp_{x}^\lambda)\cdot\Big(\frac{\partial \nu}{\partial x}(x)-\frac{\partial \nu}{\partial x}(\tx)\Big)\,dt\\
&&\hspace{16ex}\displaystyle+\int_0^{t_f}\e^{\lambda t}(p_{x}^\lambda-\tp_{x}^\lambda)\cdot\Big((p^{\lambda}_{x_1}-p^{\lambda}_s)s^{\lambda}\beta-(\tp^{\lambda}_{x_1}-\tp^{\lambda}_s)\ts^{\lambda}\tbeta\Big)\,dt\\
&&\hspace{16ex}\displaystyle+\int_0^{t_f}(p^{\lambda}_{x}-\tp^{\lambda}_{x})\cdot M^T(p_{x}^\lambda-\tp_{x}^\lambda)-(p_{x}^\lambda-\tp_{x}^\lambda)
\cdot(p^{\lambda}_r-\tp^{\lambda}_r)\sigma\,dt
\end{eqnarray*}
Let us now estimate the right hand sides. Concerning the first equation, since 
$$
s^{\lambda}\beta\cdot x^{\lambda}-\ts^{\lambda}\tbeta\cdot \tx^{\lambda}=(s^{\lambda}-\ts^{\lambda})\beta\cdot x^{\lambda}+
\ts^{\lambda} (\beta-\tbeta)\cdot x^{\lambda}+\ts^{\lambda}\tbeta\cdot
(x^{\lambda}-\tx^{\lambda})
$$
and since the states, the costates and the controls are bounded (see Remark \ref{bas}), then by triangular and Young  inequalities we have  
that there exists a positive constant $D_{11}$ such that
\begin{eqnarray*}
&&\displaystyle\Big|\int_0^{t_f}\e^{\lambda t}(x_1^\lambda-\tx_1^\lambda)(s^{\lambda}\beta\cdot x^{\lambda}-\ts^{\lambda}\tbeta\cdot \tx^{\lambda})\,dt\Big|
\le\\ 
&&\hspace{8ex}\displaystyle\le D_{11}\e^{\lambda t_f}\Big(\int_0^{t_f}|x^\lambda-\tx^\lambda|^2+|s^{\lambda}-\ts^{\lambda}|^2+
|u-\tu|^2\,dt\Big)
\end{eqnarray*}
where we used also the fact that $\beta-\tbeta=u-\tu$. On the other hand, using the characterization of the optimal control we get
\begin{eqnarray*}
\int_0^{t_f}|u-\tu|^2\,dt&\le&\sum_{i=1}^n\int_0^{t_f}|\psi_i^+(t)^\frac{1}{q_i-1}-\tpsi_i^+(t)^\frac{1}{q_i-1}|^2\,dt\\
&\le&D_{12}\sum_{i=1}^n\int_0^{t_f}|\psi_i^+(t)-\tpsi_i^+(t)|^2\,dt\\
&\le&D_{12}\sum_{i=1}^n\int_0^{t_f}|\psi_i(t)-\tpsi_i(t)|^2\,dt\\
&\le& D_{12}\int_0^{t_f}|p_x^\lambda-\tp_x^\lambda|^2+|p_s^\lambda-\tp_s^\lambda|^2+|x^\lambda-\tx^\lambda|^2+|s^\lambda-\ts^\lambda|^2\,dt
\end{eqnarray*}
for a suitable positive constant $D_{12}$ (possibly changing line by line). We used here the assumption $q_i\le 2$ and the local Lipschitz continuity of the power function $y^\frac{1}{q_i-1}$ ($y\ge0$) together with the boundedness of states and costates.

Putting together with the previuos one and estimating the other term of the equation in an analogous way, we end up with the existence of a positive constant $D_1$ such that
\begin{eqnarray*}
&&\dfrac{\big|x^\lambda(t_f)-\tx^\lambda(t_f)\big|^2}{2}+\lambda \int_0^{t_f}|x^\lambda-\tx^\lambda|^2dt\\
&&\hspace{8ex}\le D_1\e^{\lambda t_f}\Big(\int_0^{t_f}|p_x^\lambda-\tp_x^\lambda|^2+|p_s^\lambda-\tp_s^\lambda|^2+|x^\lambda-\tx^\lambda|^2+|s^\lambda-\ts^\lambda|^2\,dt\Big).
\end{eqnarray*}
To estimate the right hand side of the second equation we use assumption \eqref{lcnu} and obtain that there exist positive constants $D_2$ and $E_2$ such that
\begin{eqnarray*}
&&\displaystyle\frac{\big|p_{x}^\lambda(0)-\tp_{x}^\lambda(0)\big|^2}{2}+\lambda \int_0^{t_f}|p_{x}^\lambda-\tp_{x}^\lambda|^2dt\le\\ 
&&\hspace{8ex}\displaystyle\le D_2\e^{\lambda t_f}\int_0^{t_f}|p_{x}^\lambda-\tp_{x}^\lambda|^2+|{x}^\lambda-\tx^\lambda|^2+
|p_{s}^\lambda-\tp_{s}^\lambda|^2+|{s}^\lambda-\ts^\lambda|^2\,dt\\
&&\hspace{11ex}\displaystyle+E_2\int_0^{t_f}|p_{r}^\lambda-\tp_{r}^\lambda|^2dt.
\end{eqnarray*}
Doing analogous estimates with the other two couples of state/costate equations, and summing up, we get
\begin{eqnarray*}
&&\hspace{-4ex}\frac12\big|s^\lambda(t_f)-\ts^\lambda(t_f)\big|^2+\dfrac12\big|x^\lambda(t_f)-\tx^\lambda(t_f)\big|^2+
\frac12\big|r^\lambda(t_f)-\tr^\lambda(t_f)\big|^2\\
&&\hspace{0ex}+\frac12\big|p_s^\lambda(0)-\tp_s^\lambda(0)\big|^2+
\dfrac12\big|p_x^\lambda(0)-\tp_x^\lambda(0)\big|^2+
\dfrac12\big|p_r^\lambda(0)-\tp_r^\lambda(0)\big|^2\\
&&\hspace{0ex}+\lambda \Big(\int_0^{t_f}|s^\lambda-\ts^\lambda|^2+|x^\lambda-\tx^\lambda|^2+|r^\lambda-\tr^\lambda|^2+\\
&&\hspace{22ex}+|p_s^\lambda-\tp_s^\lambda|^2+|p_x^\lambda-\tp_x^\lambda|^2+|p_r^\lambda-\tp_r^\lambda|^2 dt\Big)\\
&&\hspace{-2ex}\le (D\e^{\lambda t_f}+E)
\cdot\Big(\int_0^{t_f}|s^\lambda-\ts^\lambda|^2+|x^\lambda-\tx^\lambda|^2+|r^\lambda-\tr^\lambda|^2\\
&&\hspace{22ex}
+|p_s^\lambda-\tp_s^\lambda|^2+|p_x^\lambda-\tp_x^\lambda|^2+|p_r^\lambda-\tp_r^\lambda|^2 dt\Big)
\end{eqnarray*}
for suitable positive constants $D$ and $E$.
This implies that 
\begin{eqnarray*}
&&(\lambda-D\e^{\lambda t_f}-E)\Big(\int_0^{t_f}|s^\lambda-\ts^\lambda|^2+|x^\lambda-\tx^\lambda|^2+|r^\lambda-\tr^\lambda|^2\\
&&\hspace{22ex}+|p_s^\lambda-\tp_s^\lambda|^2+|p_x^\lambda-\tp_x^\lambda|^2+|p_r^\lambda-\tp_r^\lambda|^2 dt\Big)\le0
\end{eqnarray*}
for every $\lambda\ge0$. By choosing $\lambda$ such that
$\lambda\ge D+E$ and 
$$
t_f<\frac{1}{\lambda}\ln\big(\dfrac{\lambda-E}{D}\big)
$$ 
we obtain that $\lambda-D\e^{\lambda t_f}-E>0$ and 
this implies that the integral is zero and therefore the two solutions are equal. 
\end{proof}

\begin{remark}\label{unique}{
It is important to remark that this is not a local uniqueness result, but a global result that holds for a small $t_f$. 
Indeed, the proof essentially relies on the transversality boundary conditions $p_s(t_f)=p_x(t_f)=p_r(t_f)=0$.
If the integration would be performed in an interval $[0,T]$ with $T\ne t_f$ then the proof was not work 
because, in general, the costates do not vanish in $T$.
This makes impossible to extend the result besides the time $t_f$ by proving it in $[0,T]$ and using the values of states and costates in $T$ to iterate the procedure. The uniqueness of the solution for every $t_f$ is still an open problem.
On the other hand, however, uniqueness is quite secondary with respect to having a global
optimal solution. }
\end{remark}

\section{The case of a linear cost in the control variable}\label{slu}

Let us consider the case in which the running cost is linear in the control variable, that is
$$
f_0=\nu(t,x)+C\cdot u
$$
where $\nu\in C^1([0,t_f]\times[0,1]^n)$ is a nonnegative function, and $C$ is a vector of strictly positive constants.
It is, in fact, like that of the previous section but with $q_i=1$, $i=1,\dots,n$.  

The Hamiltonian is  
\begin{eqnarray*}
&&\hspace{-5ex}H(t,u,s,x,r,p_s,p_x,p_r)
\displaystyle=\nu(t,x)+\big[C-(p_{x_1}-p_s) sx\big]\cdot u+\\
&&\hspace{21ex}+(p_{x_1}-p_s) s\bb\cdot x+\rho (p_s-p_r) r+ p_x\cdot M x+p_r\,\sigma\cdot x\,.
\end{eqnarray*}
Being linear with respect to $u$ with a coefficient with an unknown sign, the minimum value on 
$K=\prod_{i=1}^n[0,\ubar_i]$ is achieved when $u_i\in\{0,\ubar_i\}$, $i=1,\dots,n$.
 Hence, setting the {\em switching function} 
$$
\psi:=(p_{x_1}-p_s)sx
$$
the optimal controls have to satisfy
\begin{equation}\label{cuSIRlinu}
u_i(t)=\begin{cases}
0&\mbox{ if }\psi_i(t)< C_i,\\[0ex]
\ubar_i&\mbox{ if }\psi_i(t)>C_i.
\end{cases}
\end{equation}
Since, by Pontryagin's theorem, $\psi$ is a (absolutely) continuous function, then we have that 
\begin{itemize}
\item if $|\{t\in I\ :\ \psi_i(t)= C_i\}|=0$ then the optimal control $u_i$ is {\em bang-bang}, that is it takes essentially only the maximum and minimum values,
\item if, on the contrary, $|\{t\in I\ :\ \psi_i(t)= C_i\}|>0$ then there could exist an interval $(t_1,t_2)$, with $0\le t_1<t_2\le t_f$ such that
$\psi(t)=C_i$ for every $t\in(t_1,t_2)$ and the control is called {\em singular} and it is known that they may or may not be minimizing (see \cite[Chapter 8]{BHbook}). In principle, the existence of such an interval $(t_1,t_2)$ is not guaranteed because of the existence of compact sets of positive Lebesgue measure and empty interior; if $K\subset[0,t_f]$ is such a set, then, letting $\psi(t)$ the Euclidean distance between $t$ and $K$ we have that the Lipschitz continuous function $\psi$ is zero on $K$ and strictly positive outside.
\end{itemize}
The adjoint variables $p_s$, $p_x$ and $p_r$ must satisfy the same adjoint equations and transversality conditions \eqref{eaSIRsdnu} of the previous case.

\begin{remark}{ Let us remark that, by the transversality conditions,  $\psi(t_f)=0$. This fact, together with $C_i>0$ and the continuity of $\psi$  implies that $\psi_i(t)<C_i$ in a left neighborhood of $t_f$ ($i=1,\dots,n$) and hence $u(t)=0$ in this neighborhood. 
The optimal strategy towards the end of the epidemic horizon is then to disactivate the control policy.}
\end{remark}

\subsection{Study of the singular arcs} In the intervals in which
$\psi_i=C_i$, the control $u_i$ disappears from the expression of the Hamiltonian.
Hence, the application of Pontryagin's theorem does not give, in such intervals, any information on the optimal control that, nevertheless, is elsewhere characterized by \eqref{cuSIRlinu}. The study of singular arcs, that is of what happens in such intervals, is essential to understand the structure of the solutions. 
The reader interested into a general theory is referred to the monographs 
\cite[Section 2.8]{SLbook}, \cite[Chapter 8]{BHbook}, \cite{BC2003}.

To avoid technicalities we assume from now on to be under strictly positive initial conditions, so that by Theorem \ref{wpth} 
we have that the optimal solutions are strictly positive in the whole of $[0,t_f]$. 
To perfom computations it is convenient to denote by $M_i$ and $M^j$ the $i$-th row and the $j$-th column of the matrix $M$, respectively.  

Along a singular arc, that is for $t\in(t_0,t_1)$ we have $\psi(t)=C$, hence $\dot\psi(t)=0$. Denoting by
$$
\eta:=p_{x_1}-p_s
$$
and using $\psi=\eta s x$ then we get 
\begin{eqnarray*}
\dot\psi&=&\dot\eta s x+\eta\dot s x+\eta s \dot x\\
&=&\dot\eta s x+\eta\big(-s\,(\bb-u)\cdot x+\rho r \big)x+\eta s\big(s\,(\bb-u)\cdot x\e_1+M x\big)
\end{eqnarray*}
Since 
\begin{eqnarray}\label{etaprime}
\dot\eta&=&\dot p_{x_1}-\dot p_s=
-\frac{\partial \nu}{\partial x_1}-\eta s(\bb_1-u_1)-M^1\cdot p_{x}-p_r\sigma_1
+\eta(\bb-u)\cdot x
\end{eqnarray}
then, substituting,
$$
\dot\psi
=\Big(-\frac{\partial \nu}{\partial x_1}-\eta s(\bb_1-u_1)-M^1\cdot p_{x}-p_r\sigma_1\Big) s x
+\eta\rho r x+\eta s\big(s\,(\bb-u)\cdot x\e_1+M x\big)\,.
$$
In components we have
\begin{eqnarray*}
&&\hspace{-5ex}\dot\psi_1
=-\Big(\frac{\partial \nu}{\partial x_1}+M^1\cdot p_{x}-p_r\sigma_1
\Big)s  x_1+\eta\rho r x_1+\eta sM_{1}\cdot x+\eta s^2\sum_{j=2}^n(\bb_j-u_j)x_j \,,\\
&&\hspace{-5ex}\dot\psi_i
=-\Big(\frac{\partial \nu}{\partial x_1}+\eta s(\bb_1-u_1)+M^1\cdot p_{x}+p_r\sigma_1
\Big) s x_i+\eta\rho r x_i+\eta s M_{i}\cdot x,\ i=2,\dots,n\, .
\end{eqnarray*}
We observe that $u_1$ explicitly appears only in the expression of $\dot \psi_i$, $i=2,\dots,n$, while the other controls
appear only in the espression of $\dot \psi_1$.

\vspace{2ex}

{\bf The case $\boldsymbol{n\ge 2}$.} Along the singular arcs we have $\eta\ne0$ (since $\psi\ne0$). 
Since it is continuous then it takes a constant sign. Then we can solve the equations $\dot\psi_i=0$ for $i=2,\dots,n$ 
with respect to $\bb_1-u_1$ and obtain the feedback control laws
\begin{equation}\label{fcli}
\bb_1-u_1=\dfrac{-\Big(\frac{\partial \nu}{\partial x_1}+M^1\cdot p_{x}+p_r\sigma_1
\Big) s x_i+\eta\rho r x_i+\eta s M_{i}\cdot x}{\eta s^2 x_i}
\end{equation}
which imply that {\em $u_1$ is continuous in $(t_1,t_2)$}. In the particular case $n=2$ we can say something more.

\vspace{2ex}

{\bf The case $\boldsymbol{n=2}$.} In this case
$$
\dot\psi_1
=-\Big(\frac{\partial \nu}{\partial x_1}+M^1\cdot p_{x}-p_r\sigma_1
\Big)s  x_1+\eta\rho r x_1+\eta sM_{1}\cdot x+\eta s^2(\bb_2-u_2)x_2
$$
and  the equation $\dot\psi_1=0$ gives the feedback control law
$$
\bb_2-u_2=\dfrac{\Big(\frac{\partial \nu}{\partial x_1}+M^1\cdot p_{x}-p_r\sigma_1
\Big)s  x_1-\eta\rho r x_1-\eta s M_{1}\cdot x}{\eta s^2 x_2}.
$$
Together with \eqref{fcli} for $i=2$, that is the analogous law for $u_1$, it implies that {\em $u_1$ and $u_2$ are continuous in $(t_1,t_2)$}. Troughout the optimality system, this immediately implies more regularity also for states and costates. If $\nu$ is more regular then also the regularity of $u$ increases. Actually, we have that
if $\nu\in C^k([0,t_f]\times[0,1])$ then $u\in C^{k-1}(t_1,t_2)$. The two feedback laws can also be used to study the continuity of $u$ 
in the switching points between regions where it is constant and the singular arcs. We will do this in details in the case $n=1$.

\vspace{2ex}

{\bf The case $\boldsymbol{n=1}$.} It is the case of a SIRS model (SIR if $\rho=0$). Dropping the indication of the index one 
and setting $M=-\gamma<0$, the optimality system writes
\begin{equation*}
\begin{cases}
\dot s=-s\,(\bb-u)x+\rho r\\
\dot x=s\,(\bb-u)x-\gamma x\\
\dot r=\sigma x-\rho r\\
\dot p_s=-\eta(\bb-u) x\\
\dot p_{x}=-\frac{\partial\nu}{\partial x}-\eta s(\bb-u)+\gamma p_{x}-\sigma p_r\\
\dot p_r=\rho (p_s-p_r)\\
s(0)=s_0,\ x(0)=x_0,\ r(0)=r_0\\
p_s(t_f)=p_{x_i}(t_f)=p_r(t_f)=0
\end{cases}
\end{equation*}
Let us recall that $x_0>0$ and $s_0>0$ so that any solution 
satisfies $x(t)>0$ and $s(t)>0$ for every $t\in[0,t_f]$.
We have  
\begin{equation*}
\dot\psi
=\Big[\Big(-\frac{\partial \nu}{\partial x}+\gamma p_{s}+\sigma p_r
\Big)s +\rho \eta r\Big]x 
\end{equation*}
where the control does not explicitly appear. Since $x>0$, than the equation $\dot\psi=0$ is equivalent to
\begin{equation}\label{psidot=0}
\Big(-\frac{\partial \nu}{\partial x}+\gamma p_{s}+\sigma p_r
\Big)s +\rho \eta r =0\qquad\mbox{ in }(t_1,t_2).
\end{equation}
Assuming that $\nu$ be regular enough, differentiating \eqref{psidot=0} and putting $\bb-u$ into evidence, we get
\begin{eqnarray*}
0&\!\!=\!\!&(\bb-u)\Big[-\frac{\partial^2 \nu}{\partial x^2}s^2x 
-\gamma\eta sx+
\rho\eta r (2x-s)
\Big] \\
&&+\Big(\gamma \frac{\partial^2 \nu}{\partial x^2} x -\dfrac{\partial}{\partial t}\frac{\partial \nu}{\partial x} +\rho\sigma (p_s-p_r)
\Big)s+       \rho \Big(\gamma (p_{x}+p_s)-2\frac{\partial \nu}{\partial x}
\Big)r+\rho\eta (\sigma x-\rho r).
\end{eqnarray*}

\vspace{2ex}

{\bf SIR epidemic.} This espression becomes simpler in the case $\rho=0$, that is for an SIR epidemic with immunization, 
\begin{equation}\label{r=0fcl}
(\bb-u)x\big(s\frac{\partial^2 \nu}{\partial x^2}+\gamma\eta\big)=\gamma x\frac{\partial^2 \nu}{\partial x^2}-\dfrac{\partial}{\partial t}\frac{\partial \nu}{\partial x} .
\end{equation}

\begin{theorem}\label{thregu}
If $\nu\in C^2([0,t_f]\times[0,1])$ and 
\begin{equation}\label{nucon}
\gamma x\frac{\partial^2 \nu}{\partial x^2}(t,x)-\frac{\partial^2 \nu}{\partial t\partial x}(t,x)>0
\end{equation}
for every $t\in[0,t_f]$ and $x\in[0,1]$, then the following feedback control law holds
\begin{equation}\label{fcl2t}
u(t)=\bb-\frac{\displaystyle\gamma x(t)\frac{\partial^2 \nu}{\partial x^2}(t,x(t))-\frac{\partial^2 \nu}{\partial t\partial x}(t,x(t))}{\displaystyle x(t)\Big(s(t)\frac{\partial^2 \nu}{\partial x^2}(t,x(t))+\gamma\eta(t)\Big)}
\end{equation}
for every $t\in(t_1,t_2)$. Moreover,  
\begin{enumerate}
\item $u$ is continuous in $(t_1,t_2)$ and there exist, and are finite, the right and the left limits of $u$ in $t_1$ and $t_2$, respectively;
\item let $k\in\N\cup\{\infty\}$, $k\ge2$; if $\nu\in C^k([0,t_f]\times[0,1])$ then $u\in C^{k-2}(t_1,t_2)$.
\end{enumerate}
\end{theorem}

\begin{proof} Since the right hand side of \eqref{r=0fcl} is strictly positive, and since $\bb-u\ge 0$, this means that 
$$
\bb-u>0\ \mbox{ and }\ x\big(s\frac{\partial^2 \nu}{\partial x^2}+\gamma\eta\big)>0\ \mbox{ in }(t_1,t_2).
$$
Solving for $u$ we find \eqref{fcl2t}. {\it1.} follows from the continuity in $[0,t_f]$ of the states, the costates and 
the second derivatives of $\nu$. This proves also {\it2.}\ in the case $k=2$ and the optimality system implies that states and costates belongs to $C^1(t_1,t_2)$. 

{\it2.}\ follows by induction on $k$ observing that if $\nu\in C^{k+1}$ and states and costates are $C^{k-1}(t_1,t_2)$ then
 \eqref{fcl2t} implies $u\in C^{k-1}(t_1,t_2)$. 
\end{proof}

\begin{remark}{
Assumption \eqref{nucon} is clearly satisfied if $\nu$ is strictly convex and independent of $t$ and, in such case, the feedback law  takes the even simpler form 
\begin{equation}\label{fcl2}
u(t)=\bb-\frac{\displaystyle\gamma \frac{\partial^2 \nu}{\partial x^2}\big(x(t)\big)}{\displaystyle s(t)\frac{\partial^2 \nu}{\partial x^2}\big(x(t)\big)+\gamma\eta(t)},\qquad t\in (t_1,t_2). 
\end{equation}}
\end{remark}

\begin{remark}\label{nulin}{
The feedback law is a necessary condition for the existence of a singular arc. 
A case in which it cannot be satisfied is when $\nu$ is linear and $t$-independent. Indeed, in such a case we have that the second derivatives identically vanish in \eqref{fcl2} and the law gives $u=\bb$. If $\ubar<\bb$ then it cannot be satisfied and singular arcs do not exist. If $\ubar=\bb$ then we have $u=\ubar$ in $(t_1,t_2)$ and the optimal control is piecewise constant. It has been proved in \cite{KS2020} that when $\nu$ is linear the optimal control must be quasi-concave (that is first increasing and then decreasing), then we can conclude that it is piecewise constant and can switch in at most two points (according to Propositions 6 of \cite{KS2020}). See Figure \ref{L-L_mub=008}.
}
\end{remark}

\begin{remark}
Under the assumptions of Theorem \ref{thregu},
at the switching points between a region in which the control is constant and a singular arc there exist the right and left limits of $u$. The control turns out to be continuous if and only if these limits match the constant values of the control outside the singular arc. 
\end{remark}

\vspace{2ex}

{\bf SIR epidemic with an autonomous cost functional.}
When $\nu$ is independent of time
then the Hamiltonian is constant along the optimal solutions, that is, there exists a constant $k$ such that 
\begin{equation}\label{H=KSIRlin}
\nu(x)+Cu+\eta s(\bb-u) x- \gamma p_x x=k
\end{equation}
on the whole interval $[0,t_f]$. Computing in $t_f$, using the transversality conditions and since, as already observed, $u(t_f)=\eta(t_f)=0$, then we have
$$
k=\nu(x(t_f)).
$$ 
Equation \eqref{H=KSIRlin} can be used, together with the adjoint equations that give (see \eqref{etaprime})
$$ 
\dot\eta=
-\frac{\partial\nu}{\partial x}(x)+\eta(\bb-u)(x-s)+\gamma p_{x}-\sigma p_r,
$$
to find another differential equation for $\eta$.
Indeed, by \eqref{H=KSIRlin} we have
$$
\eta(\bb-u){s}=\frac{\nu(x(t_f))-\nu(x)-Cu}{x}+\gamma p_x
$$
and substituting into the expression of $\dot\eta$ we get
\begin{equation}\label{echc}
\dot\eta=
\eta({\bb}-u){x}+\frac{\nu(x)+Cu-\nu(x(t_f))-\nu'(x)x}{x}.
\end{equation}
The usage of $\eta$ is quite natural. Nevertheless, the idea that two adjoint variables can be summarized into a
single new variable is already in \cite{Behncke2000} and used also in \cite{KS2020} where the following proposition is
proved under assumption {\it2}.

\begin{proposition}\label{etapos} For every $t\in[0,t_f]$ we have
\begin{enumerate}
\item if $\nu$ is nondecreasing then $\eta(t)\ge0$,
\item if $\nu$ is strictly increasing then $\eta(t)>0$.
\end{enumerate}
\end{proposition}

\begin{proof} Arguing by contradiction, let us suppose that there exists $t\in[0,t_f]$ such that corresponding to the two cases of the statement, 
\begin{itemize}
\item[\it{1.}] $\eta(t)<0$, 
\item[\it{2.}] $\eta(t)\le0$.
\end{itemize} 
Since the switching function $\psi=\eta s x$ takes the same sign as $\eta$, and 
since $C>0$, in both cases we have $\psi(t)<C$, hence $u(t)=0$. On the other hand, since $\psi$ is continuous, then $\psi<C$, 
and hence $u=0$, in a neighborhood of $t$. Using \eqref{echc} and the fact that $\nu$ is increasing and convex, 
repeating the argument of \cite{KS2020},
in this neighborhood we have
{\begin{eqnarray*}
\dot\eta&=
&\eta\bb{x}+\frac{\nu(x)-\nu(x(t_f))-\nu'(x)x}{x}\\
&\le&\eta\bb{x}+\frac{\nu(i)-\nu(x(t_f))-\nu'(x)x+\nu'(x)x(t_f)}{x}\\
&=&\eta\bb{x}+\frac{\nu(x)-\nu(x(t_f))-\nu'(x)(x-x(t_f))}{x}\\
&\le&\eta\bb{x},
\end{eqnarray*}
and the strict inequality holds if $\nu$ is strictly increasing since, in this case, we have $\nu'(x)x(t_f)>0$. 

In both cases then we have
$$
\dot\eta(t)<0.
$$
This would imply that $\eta(s)<0$ for every $s>t$, which contradicts the fact that $\eta(t_f)=0$. }

\end{proof}

Proposition \ref{etapos} has consequences regarding the effectiveness of the control policies.

\begin{proposition}\label{monsa} Let $\nu$ be of class $C^2$.
\begin{enumerate}
\item If $\nu$ is  convex and nondecreasing then, along the singular arcs, the population of infected individuals weakly decreases.
\item Let $\ubar<\bb$. If $\nu$ is strictly convex and strictly increasing then, along the singular arcs, the population of infected individuals strictly decreases.
\end{enumerate}
\end{proposition}

\begin{proof} By \eqref{psidot=0}, in the autonomous case with $\rho=0$, we have
$$
-\nu'(x)+\gamma p_{s}=0 \qquad\mbox{ in }(t_1,t_2).
$$ 
Computing the first derivative and using the adjoint equation {$\dot{p}_s=-\eta(\bb-u)x$}, we have 
$$
\nu''(x)\dot x=-\gamma \eta(\bb-u)x\quad\mbox{ in }(t_1,t_2).
$$
Since moreover $\gamma x>0$, then
\begin{itemize} 
\item under assumption {\it1.}\ we have $\eta\ge0$, $\nu''\ge0$ and $\bb-u\ge0$; 
hence $\dot x\le0$ and $x$ is nonincreasing;
\item under assumption {\it2.}\ we have $\eta>0$, $\nu''>0$ and $\bb-u\ge0$; hence $\dot x<0$ and $x$ is strictly decreasing.
\end{itemize}
\end{proof}

\begin{remark}\label{rtu}{
The proof of Proposition \ref{etapos} works also for a running cost of the form $f_0=\nu(x)+Cu^q$ with $q>1$ like in Section \ref{sssup}, leading, in the case of a nondecreasing $\nu$, to $\psi=\frac{1}{qC}\eta s x\ge0$ and, hence,  to the
 following simpler characterization of the optimal control
\begin{equation}\label{utrunc}
u(t)=\min\{\psi(t),\ubar\}.
\end{equation}
See Figure \ref{Q-Q_mub=008}. A case in which $\psi(t)\le\ubar$ for every $t$ is shown in Figure \ref{Q-Q_mub=004}.}
\end{remark}

\begin{remark}[Behavior at the switching points]\label{bsp}{
We have already remarked that if $\nu\in C^2$ is convex and independent of $t$ then the assumptions of Theorem \ref{thregu}
are satisfied. Then, 
at the switching points between a region in which the control is constant and a singular arc, the control turns out to be continuous if and only if the right and left limits at the extrema of the interval $(t_1,t_2)$ match the constant values of the control outside the interval. 
If, for instance, $t_1$ is a switching point between an interval in which $u$ is the constant $0$ and the singular arc
then, in the strictly convex autonomous case, 
the continuity condition is
$$
s(t_1)=\frac{\gamma}{\bb}-\frac{\gamma\eta(t_1)}{\frac{\partial^2 \nu}{\partial x^2}(x(t_1)) },
$$
which implies
$$
s(t_1)<\frac{\gamma}{\bb}.
$$
Let us remark that ${\gamma}/{\bb}$ is the number of susceptible individuals that corresponds to the uncontrolled epidemic peak.
Since it is convenient to activate the control before the peak time (if it not identically zero and since otherwise a translation of the control 
function would provide a better performance) then we expect to have always a discontinuity at the first switching time like in Figure
\ref{Q-L_mub=01} and \ref{Q-L_mub=008}.
If, instead, $t_1$ is a switching point between an interval in which $u$ is the constant $\ubar$ and the singular arc, then
the continuity condition is
$$
\bb-\ubar=\frac{\displaystyle\gamma \frac{\partial^2 \nu}{\partial x^2}\big(x(t_1)\big)}{\displaystyle s\frac{\partial^2 \nu}{\partial x^2}\big(x(t_1)\big)+\gamma\eta(t_1)}.
$$
We deduce that, if $\ubar=\bb$ then, in the strictly convex autonomous case, the optimal control is always discontinuous in this kind of switching points. Such kind of discontinuites occur in  Figure \ref{Q-L_mub=01} and \ref{Q-L_mub=008} .
 }
\end{remark}

\section{Bocop simulations} 

To conclude, we present some numeric simulations done by using the Bocop package, \cite{Bocop, BocopExamples}.
We do not aim here to perform numerical analysis, but just use them as examples to explain some results. 
For this reason, and for simplicity,
the simulations are made on the SIR epidemic model 
\begin{equation*}
\begin{cases}
\dot s=-s\,(\bb-u)x\\
\dot x=s\,(\bb-u)x-\gamma x\\
s(0)=s_0,\ x(0)=x_0\,.
\end{cases}
\end{equation*}
We consider the following three cost functionals with different growths in the state and control variables that are paradigmatic 
of the analysis performed in Section \ref{sssup} and \ref{slu}:
\begin{itemize}
\item $\displaystyle J_{QQ}(x,u)=\int_0^{t_f} \big(x^2+u^2\big)\,dt$, quadratic in state and control;
\item $\displaystyle  J_{QL}(x,u)=\int_0^{t_f} \big(30 x^2+ u\big)\,dt$, quadratic in state and linear in the control; 
\item $\displaystyle J_{LL}(x,u)=\int_0^{t_f} \big(2x+ u\big)\,dt$, linear in state and control.
\end{itemize}
The first functional falls in the theory devoloped in Section \ref{sssup}, while the others refer to Section \ref{slu}.

\newpage

The Bocop package implements a local optimization method. The optimal control problem is approximated by a finite dimensional optimization problem (NLP) using a time discretization (the direct transcription approach). The NLP problem is solved by the well known software Ipopt, using sparse exact derivatives computed by CppAD. 
The default list of discretization formulas proposed by the package includes: Euler, Midpoint, Gauss II and Lobatto III C. Among them, we have chosen to use Lobatto III C for its numerical stability.
Indeed, it is well known that it is an excellent method for stiff problems (see \cite{J2015}) like the computation of singular arcs. Using it, we have avoided some numerical instabilities
developed by the other methods in such kind of computations.

We consider a time horizon $t_f$ of $360$ days. The choice of the coefficients $\bb=0.16$, $\gamma=0.06$ and of the initial conditions $i_0=0.001$, $s_0=0.999$,
has been done according to \cite{KS2020}. The coefficients in front of the state in the cost functionals are choosen in a way to balance  the contributions of the two terms and ensure convergence of the computations.

 In Figure \ref{Q-Q_mub=008} and \ref{Q-Q_mub=004} the cost is quadratic both in the state and in the control variables. 
In the second, the maximum value of the control would exceed the upper bound $\ubar$ and then it is truncated 
according to Remark \ref{rtu} and equation \eqref{utrunc}.


\begin{figure}[H]
\begin{minipage}{0.49\textwidth}
\begin{center}
\includegraphics[width=\textwidth]{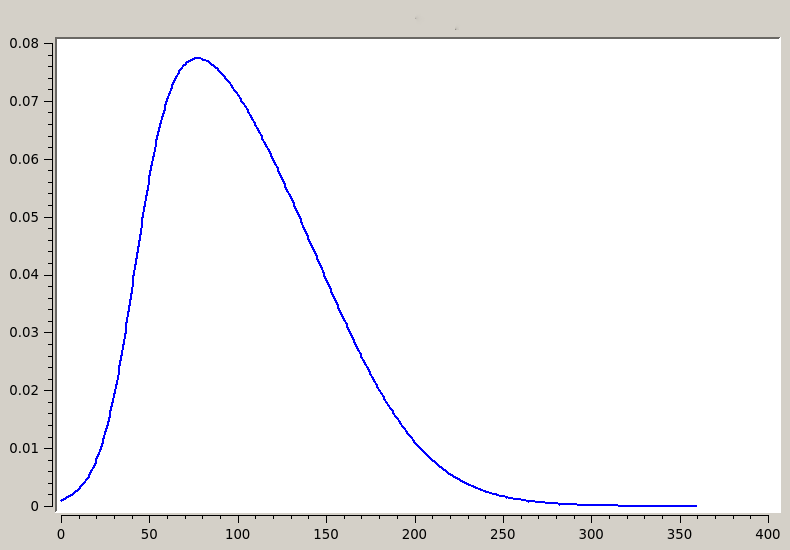}\\
optimal control $u$
\end{center}
\end{minipage}
\begin{minipage}{0.49\textwidth} 
\begin{center}
\includegraphics[width=\textwidth]{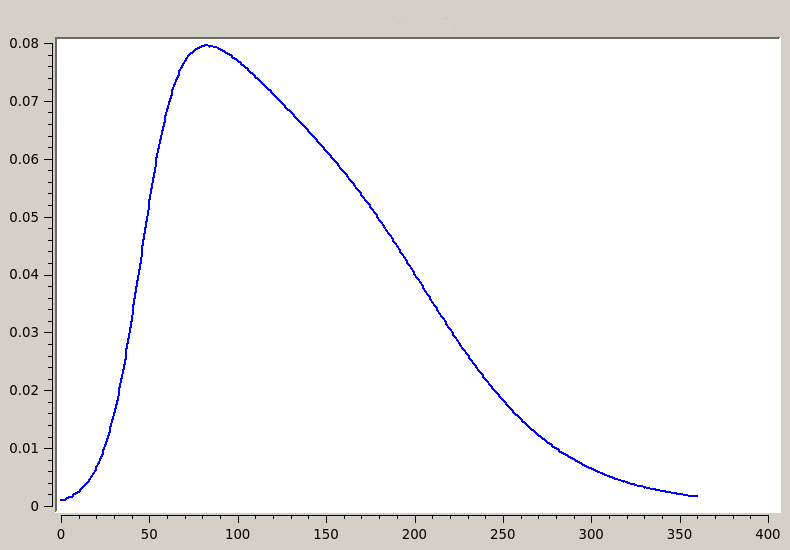}\\
state $x$ (infected)
\end{center}
\end{minipage}
\vspace{2ex}
\caption{$J_{QQ}$ with $\ubar=0.08$}
\label{Q-Q_mub=008}
\end{figure}

\vspace{-2ex}

\begin{figure}[H]
\begin{minipage}{0.49\textwidth}
\begin{center}
\includegraphics[width=\textwidth]{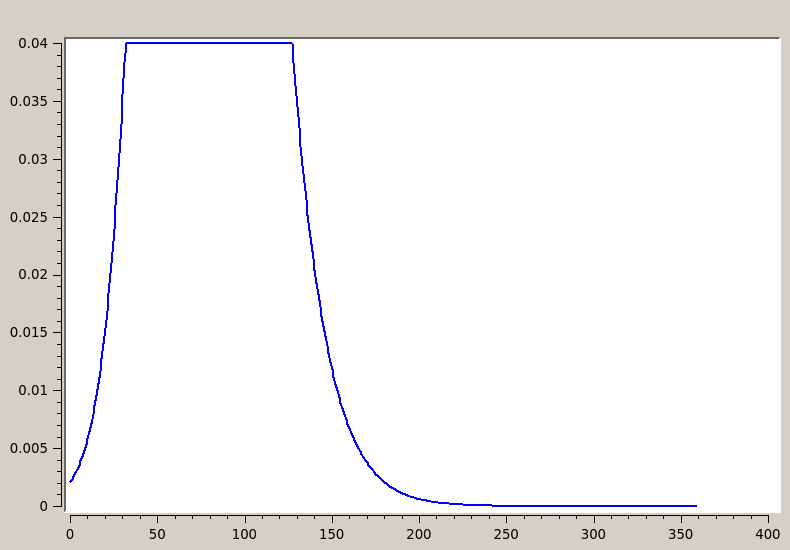}\\
optimal control $u$
\end{center}
\end{minipage}
\begin{minipage}{0.49\textwidth} 
\begin{center}
\includegraphics[width=\textwidth]{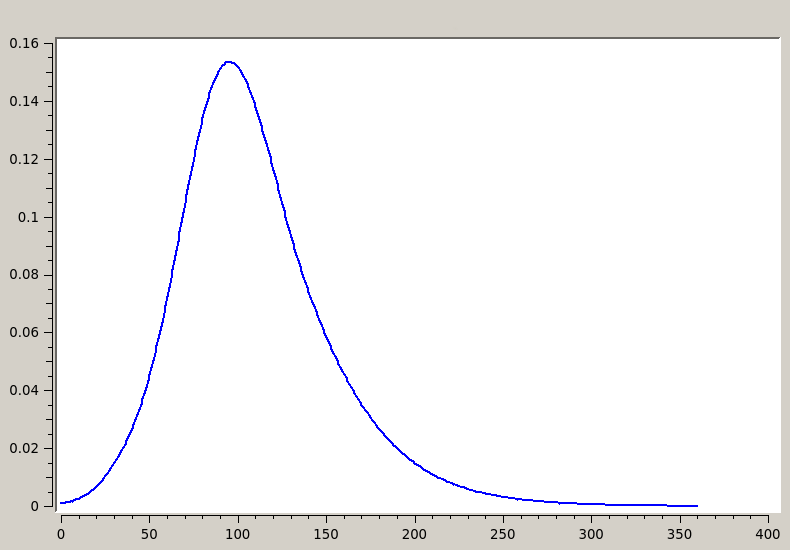}\\
state $x$ (infected)
\end{center}
\end{minipage}
\vspace{2ex}
\caption{$J_{QQ}$ with $\ubar=0.04$}
\label{Q-Q_mub=004}
\end{figure}

\newpage

In Figure \ref{Q-L_mub=01} and \ref{Q-L_mub=008} the cost is quadratic in the state but linear in the control and therefore singular arcs can be expected. In fact, Figure \ref{Q-L_mub=01} shows a bang-singular-bang control structure, while a bang-bang-singular-bang control appears in Figure \ref{Q-L_mub=008}. Note that all discontinuities at the switching points are predicted in Remark \ref{bsp}.
Moreover it can be observed that the population of infected individuals strictly decreases along the singular arcs
as predicted by Proposition \ref{monsa}.

\vspace{6ex}

\begin{figure}[H]
\begin{minipage}{0.49\textwidth}
\begin{center}
\includegraphics[width=\textwidth]{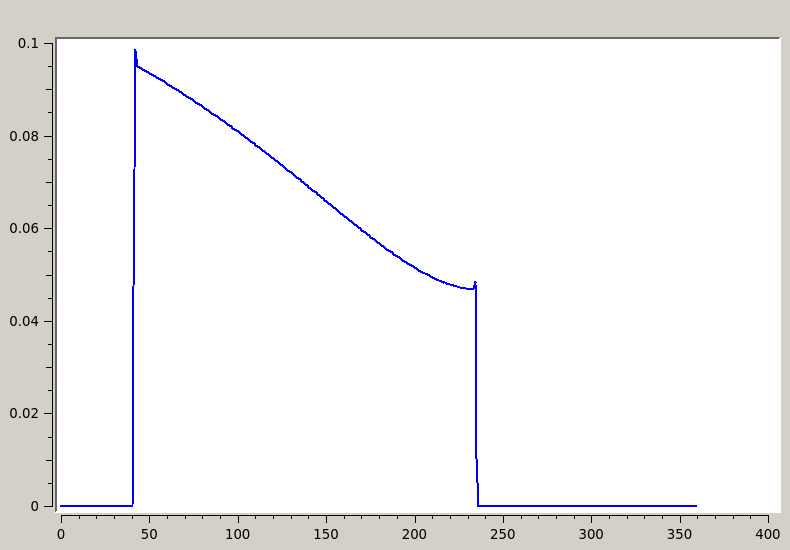}\\
optimal control $u$
\end{center}
\end{minipage}
\begin{minipage}{0.49\textwidth} 
\begin{center}
\includegraphics[width=\textwidth]{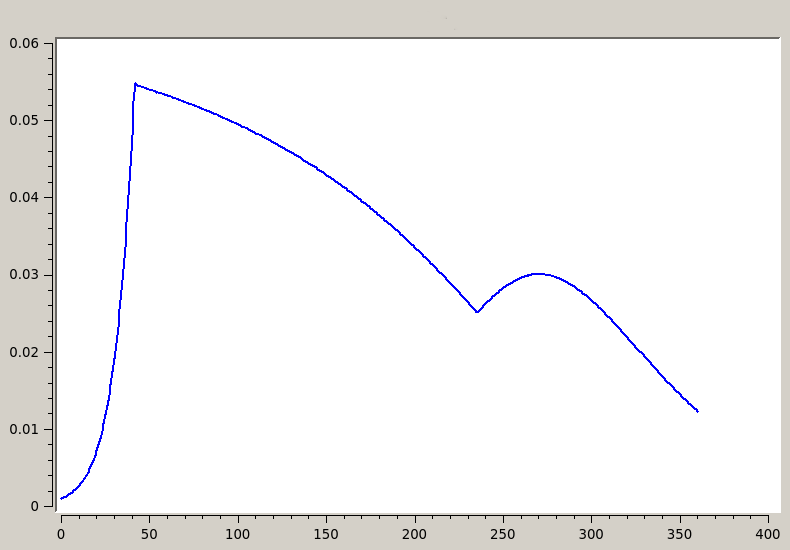}\\
state $x$ (infected)
\end{center}
\end{minipage}
\vspace{2ex}
\caption{$J_{QL}$ with $\ubar=0.1$}
\label{Q-L_mub=01}
\end{figure}

\begin{figure}[H]
\begin{minipage}{0.49\textwidth}
\begin{center}
\includegraphics[width=\textwidth]{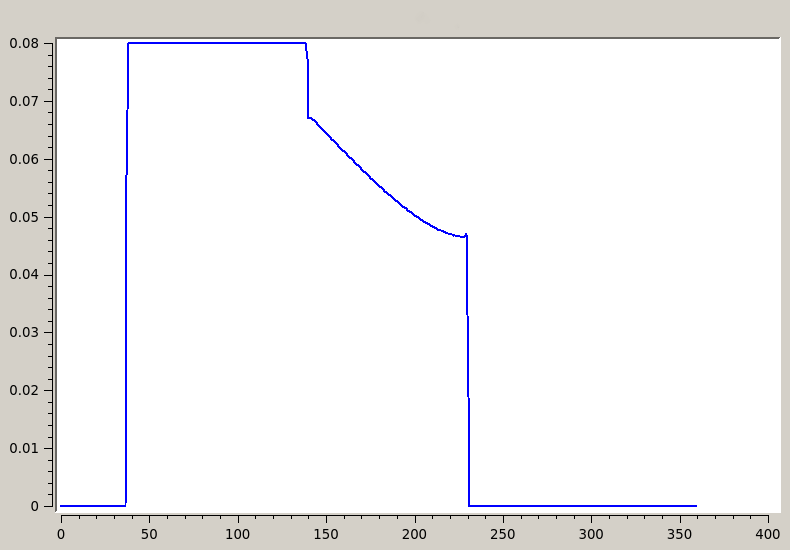}\\
optimal control $u$
\end{center}
\end{minipage}
\begin{minipage}{0.49\textwidth} 
\begin{center}
\includegraphics[width=\textwidth]{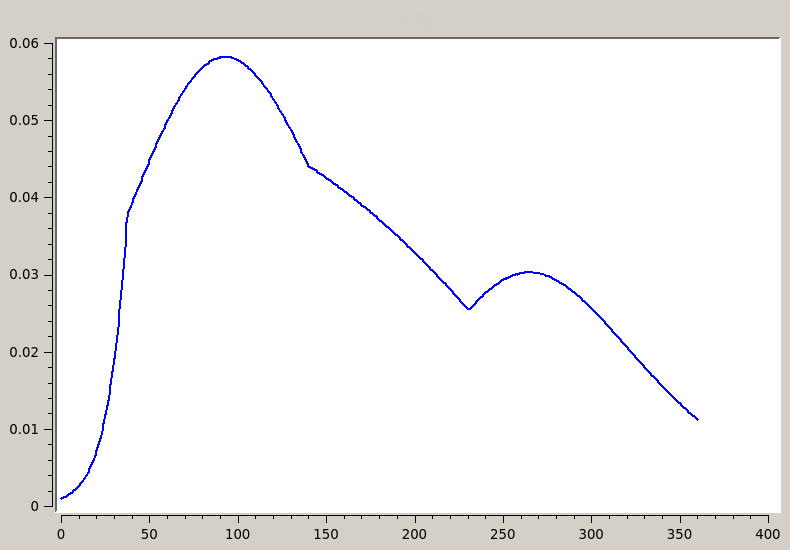}\\
state $x$ (infected)
\end{center}
\end{minipage}
\vspace{2ex}
\caption{$J_{QL}$ with $\ubar=0.08$}
\label{Q-L_mub=008}
\end{figure}

\newpage

In Figure \ref{L-L_mub=008} the cost is linear in both the state and the control variables. In this case only 
bang-bang controls with at most two switching points are permitted according to Remark \ref{nulin}.

\begin{figure}[H]
\begin{minipage}{0.99\textwidth}
\begin{center}
\includegraphics[width=0.9\textwidth]{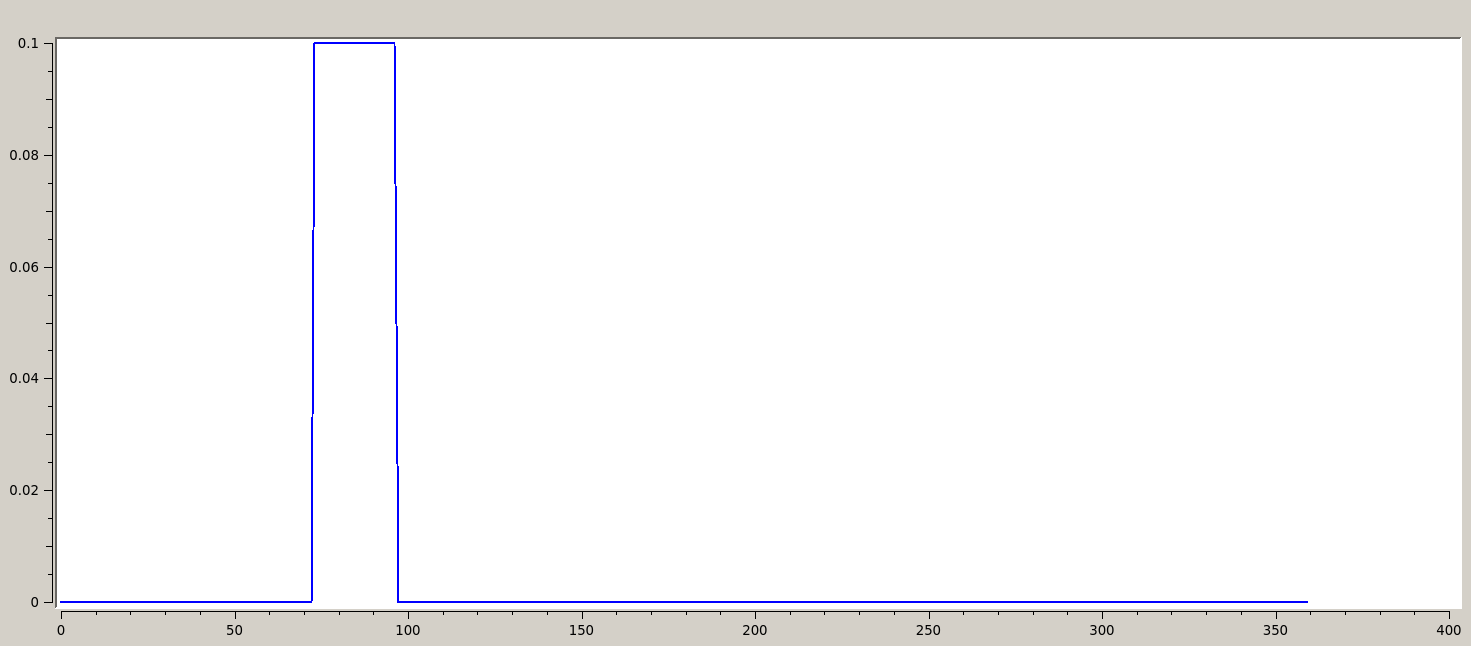}\\
optimal control $u$
\end{center}
\end{minipage}

\vspace{1ex}

\begin{minipage}{0.99\textwidth} 
\begin{center}
\includegraphics[width=0.9\textwidth]{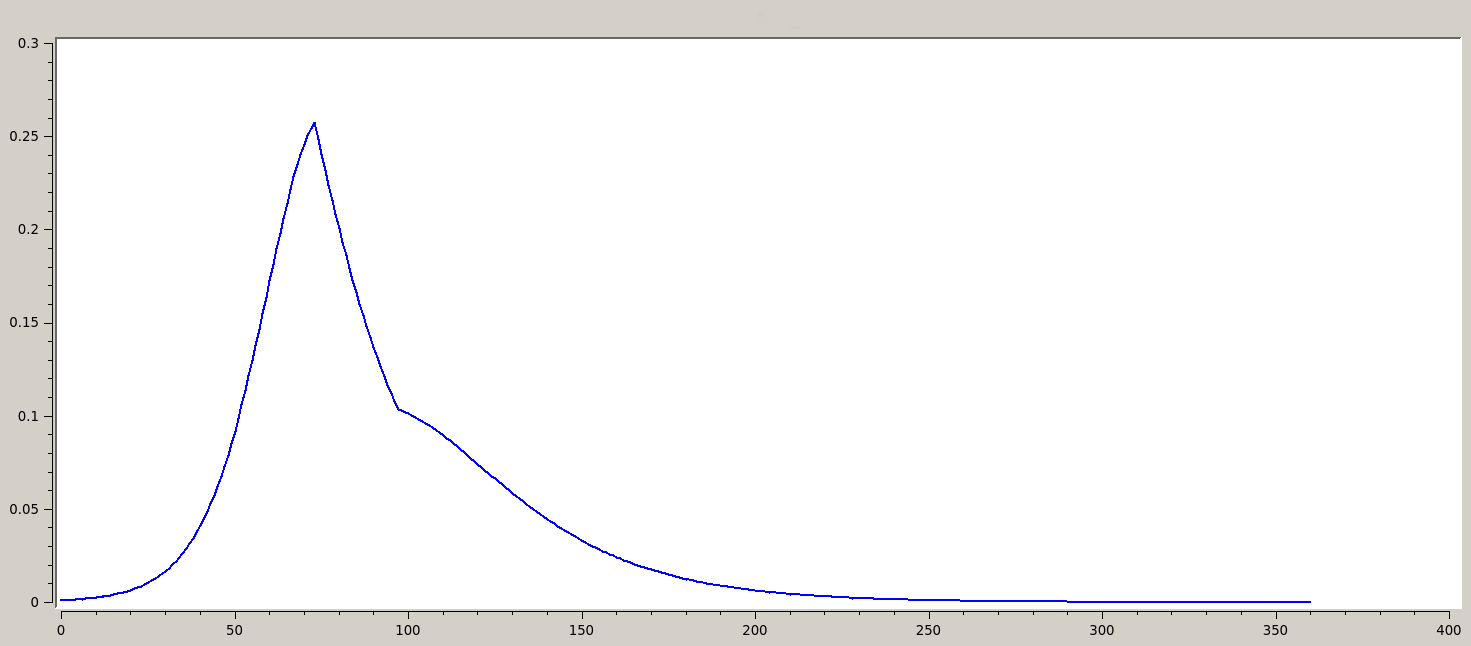}\\
state $x$ (infected)
\end{center}
\vspace{2ex}
\end{minipage}
\caption{$J_{LL}$ with $\ubar=0.1$}
\label{L-L_mub=008}
\end{figure}

\section{Conclusions and perspectives}

We introduced a general system of ordinary differential equations that accounts for a vector valued state function whose components represent various kinds of exposed/infected subpopulations, with a corresponding vector of control functions possibly different for any subpopulations. It includes some classical and recent  models for the epidemic spread in a closed population without vital dynamic in a finite time horizon.  

In the general setting, we proved well-posedness and positivity of the initial value problem for the system of state equations and
the existence of solutions to the optimal control problem of the coefficients of the nonlinear part of the system, under a very general cost functional. We also proved the uniqueness of the optimal solution for a small time horizon when the
cost is superlinear in all control variables with possibly different exponents in the interval $(1,2]$. 

In a second part of the paper we studied necessary optimality conditions.
In the case of a linear cost in the control variables, in which singular arcs are expected, we  
derived feedback control laws that allow for the study of qualitative properties of the optimal solutions
like monotonicity (Proposition \ref{monsa}) and regularity. In particular, in the quadratic case the optimal control turns out to be a Lipschitz continuous function (Proposition \ref{qcuLc}). On the contrary, when the control appears linearly
discontinuities are expected to occur between regions in which the control is constant and the singular arcs, according to the analysis 
developed in Section \ref{slu}. Finally, the results are illustrated by the aid of some numerical simulations.

For simplicity, the analysis done in Section \ref{slu} has been mainly limited to the case of a SIR model and can be further developed by considering some different or more general situations. Also the introduction of general spatial terms (reaction-diffusion like) in the state equations could be an interesting development direction.


\begin{thebibliography}{99}

\bibitem{AM1991}
M.~R.~M. Anderson, 
\newblock {\em Infectious diseases of humans}.
\newblock London: Oxford University Press, 1991.

\bibitem{Behncke2000}
H.~Behncke.
\newblock Optimal control of deterministic epidemics.
\newblock {\em Optimal Control Appl. Methods}, 21(6):269--285, 2000.

\bibitem{BJ2018}
W.~Bock and Y.~Jayathunga.
\newblock Optimal control and basic reproduction numbers for a compartmental
  spatial multipatch dengue model.
\newblock {\em Math. Methods Appl. Sci.}, 41(9):3231--3245, 2018.

\bibitem{BocopExamples}
J.~Bonnans, Frederic, D.~Giorgi, V.~Grelard, B.~Heymann, S.~Maindrault,
  P.~Martinon, O.~Tissot, and J.~Liu.
\newblock {Bocop – A collection of examples}.
\newblock Technical report, INRIA, 2017.

\bibitem{BC2003}
B.~Bonnard and M.~Chyba.
\newblock {\em Singular trajectories and their role in control theory},
  volume~40 of {\em Math\'{e}matiques \& Applications (Berlin) [Mathematics \&
  Applications]}.
\newblock Springer-Verlag, Berlin, 2003.

\bibitem{BHbook}
A.~E. Bryson, Jr. and Y.~C. Ho.
\newblock {\em Applied optimal control}.
\newblock Hemisphere Publishing Corp. Washington, D. C.; distributed by Halsted
  Press [John Wiley \& Sons], New York-London-Sydney, 1975.
\newblock Optimization, estimation, and control, Revised printing.

\bibitem{Clarke2013}
F.~Clarke.
\newblock {\em Functional analysis, calculus of variations and optimal
  control}, volume 264 of {\em Graduate Texts in Mathematics}.
\newblock Springer, London, 2013.

\bibitem{DE2000}
H.~J.~A. Diekmann, O.
\newblock {\em Mathematical epidemiology of infectious diseases: model
  building, analysis and interpretation}.
\newblock New York: Wiley, 2000.

\bibitem{LML2010}
L.~Feng, M.~Kumar, and L.~Mark.
\newblock An optimal control theory approach to non-pharmaceutical
  interventions.
\newblock {\em BMC Infectious Diseases}, 10(32):1471--2334, 2010.

\bibitem{FLMcN1998}
K.~R. Fister, S.~Lenhart, and J.~S. McNally.
\newblock Optimizing chemotherapy in an {HIV} model.
\newblock {\em Electron. J. Differential Equations}, pages No. 32, 12, 1998.

\bibitem{FL_MMCVLp}
I.~Fonseca and G.~Leoni.
\newblock {\em Modern methods in the calculus of variations: {$L^p$} spaces}.
\newblock Springer Monographs in Mathematics. Springer, New York, 2007.

\bibitem{GS2009}
H.~Gaff and E.~Schaefer.
\newblock Optimal control applied to vaccination and treatment strategies for
  various epidemiological models.
\newblock {\em Math. Biosci. Eng.}, 6(3):469--492, 2009.

\bibitem{GH2018}
S.~R. Gani and S.~Halawar.
\newblock Optimal control analysis of deterministic and stochastic epidemic
  model with media awareness programs.
\newblock {\em An International Journal of Optimization and Control: Theories
  \& Applications (IJOCTA)}, 9(1):24--35, 2018.

\bibitem{GBetal2020}
G.~Giordano, F.~Blanchini, R.~Bruno, P.~Colaneri, A.~Di~Filippo, A.~Di~Matteo,
  and M.~Colaneri.
\newblock Modelling the covid-19 epidemic and implementation of population-wide
  interventions in italy.
\newblock {\em Nat Med.}, 2020. Jun;26(6):855-860. doi: 10.1038/s41591-020-0883-7.

\bibitem{Gumel04}
{Gumel A.B.\ et al.}
\newblock Modelling strategies for controlling sars outbreaks.
\newblock {\em Proc. R. Soc. Lond. B.}, 271(1554):2223–2232, 2004.

\bibitem{Hale1980}
J.~K. Hale.
\newblock {\em Ordinary differential equations}.
\newblock Robert E. Krieger Publishing Co., Inc., Huntington, N.Y., second
  edition, 1980.

\bibitem{HD2011}
E.~Hansen and T.~Day.
\newblock Optimal control of epidemics with limited resources.
\newblock {\em J. Math. Biol.}, 62(3):423--451, 2011.

\bibitem{HR2004}
G.~Herzog and R.~Redheffer.
\newblock Nonautonomous {SEIRS} and {T}hron models for epidemiology and cell
  biology.
\newblock {\em Nonlinear Anal. Real World Appl.}, 5(1):33--44, 2004.

\bibitem{Heth2000}
H.~W. Hethcote.
\newblock The mathematics of infectious diseases.
\newblock {\em SIAM Rev.}, 42(4):599--653, 2000.

\bibitem{J2015}
L.~O. Jay. 
\newblock Lobatto methods. 
\newblock  {\em Encyclopedia of Applied and Computational Mathematics, Numerical Analysis of Ordinary Differential Equations}, 
\newblock Springer - The Language of Science, Björn Engquist (Ed.), 2015.


\bibitem{KK2014}
K.~Kandhway and J.~Kuri.
\newblock How to run a campaign: optimal control of {SIS} and {SIR} information
  epidemics.
\newblock {\em Appl. Math. Comput.}, 231:79--92, 2014.

\bibitem{KMK1927}
W.~O. {Kermack} and A.~G. {McKendrick}.
\newblock {A Contribution to the Mathematical Theory of Epidemics}.
\newblock {\em Proceedings of the Royal Society of London Series A},
  115(772):700--721, Aug. 1927.

\bibitem{KS2020}
T.~Kruse and P.~Strack.
\newblock Optimal control of an epidemic through social distancing.
\newblock {https://ssrn.com/abstract=3581295}, 2020.

\bibitem{LS2011}
U.~Ledzewicz and H.~Sch\"{a}ttler.
\newblock On optimal singular controls for a general {SIR}-model with
  vaccination and treatment.
\newblock {\em Discrete Contin. Dyn. Syst.}, (Dynamical systems, differential
  equations and applications. 8th AIMS Conference. Suppl. Vol. II):981--990,
  2011.

\bibitem{LEE2013310}
J.~Lee, J.~Kim, and H.-D. Kwon.
\newblock Optimal control of an influenza model with seasonal forcing and
  age-dependent transmission rates.
\newblock {\em Journal of Theoretical Biology}, 317:310 -- 320, 2013.

\bibitem{MK2012}
S.~Maharaj and A.~Kleczkowski.
\newblock Controlling epidemic spread by social distancing : do it well or not
  at all.
\newblock {\em BMC Public Health}, 12, 2012.


\bibitem{SLbook}
H.~Sch\"{a}ttler and U.~Ledzewicz.
\newblock {\em Geometric optimal control}, volume~38 of {\em Interdisciplinary
  Applied Mathematics}.
\newblock Springer, New York, 2012.
\newblock Theory, methods and examples.


\bibitem{SM2017}
O.~Sharomi and T.~Malik.
\newblock Optimal control in epidemiology.
\newblock {\em Ann. Oper. Res.}, 251(1-2):55--71, 2017.

\bibitem{Bocop}
I.~S. Team~Commands.
\newblock Bocop: an open source toolbox for optimal control.
\newblock {http://bocop.org}, 2017.

\bibitem{TLSB2020}
C.~Tsay, F.~Lejarza, M.~A. Stadtherr, and M.~Baldea.
\newblock Modeling, state estimation, and optimal control for the us covid-19
  outbreak. 
\newblock	{\em Sci Rep} 10711(10), 2020. doi: 10.1038/s41598-020-67459-8.

\bibitem{YZ2009}
X.~Yan and Z.~Yun.
\newblock Control of epidemics by quarantine and isolation strategies in highly
  mobile populations.
\newblock {\em International journal of information and systems sciences},
  5(3-4):271--286, 2009.

\bibitem{YZ2008}
X.~Yan and Y.~Zou.
\newblock Optimal and sub-optimal quarantine and isolation control in {SARS}
  epidemics.
\newblock {\em Math. Comput. Modelling}, 47(1-2):235--245, 2008.


\end{thebibliography}
\end{document}